\documentclass[a4paper,10pt]{article}

\usepackage{amsfonts,amssymb,amsmath}

\numberwithin{equation}{section}

\usepackage{tikz-cd}
 \usepackage{comment}
 \input{xy}
\xyoption{all}

\newtheorem{theo}{Theorem}[section]
\newtheorem{coro}[theo]{Corollary}
\newtheorem{lemm}[theo]{Lemma}
\newtheorem{prop}[theo]{Proposition}
\newtheorem{defi}[theo]{Definition}
\newtheorem{rema}[theo]{Remark}
\newtheorem{exam}[theo]{Example}

\newenvironment{proof}{\noindent \textbf{{Proof.}} \sf}

\def\qed{\hfill $\diamond$ \bigskip}

\def\A{{\mathcal A}}
\def\B{{\mathcal B}}

\def\lim{\mathop{\rm lim}\nolimits}

\def\Ext{\mathsf{Ext}}

\def\Tor{\mathsf{Tor}}

\def\Ker{\mathsf{Ker}}

\def\Im{\mathsf{Im}}

\def\dim{\mathsf{dim}}

\begin{document}
\sf

\title{Han's conjecture for bounded extensions}
\author{Claude Cibils,  Marcelo Lanzilotta, Eduardo N. Marcos,\\and Andrea Solotar
\thanks{\footnotesize This work has been supported by the projects  UBACYT 20020170100613BA, PIP-CONICET 11220150100483CO, USP-COFECUB.
The third mentioned author was supported by the thematic project of FAPESP 2014/09310-5, a research grant from CNPq 302003/2018-5  and acknowledges support from the "Brazilian-French Network in Mathematics". The fourth mentioned author is a research member of CONICET (Argentina), Senior Associate at ICTP and visiting Professor at Guangdong Technion-Israel Institute of Technology.}}

\date{}
\maketitle
\begin{abstract}
Let $B\subset A$ be a left or right bounded extension of finite dimensional algebras. We use the Jacobi-Zariski long nearly exact sequence to show that $B$ satisfies Han's conjecture if and only if $A$ does, regardless if the extension splits or not. We provide conditions ensuring that an extension by arrows and relations is left or right bounded. Finally we give a structure result for extensions of an algebra  given by a quiver and admissible relations, and examples  of non split  left or right bounded extensions.
\end{abstract}




\section{\sf Introduction}

Let $k$ be a field.  We recall  that for noetherian rings the left and right global dimension coincide, see \cite{AUSLANDER}.  Han's conjecture (see \cite{HAN}) states that for a finite dimensional algebra $A$, if $H_*(A,A)=0$ in large enough degrees, then $A$ has finite global dimension.  In  \cite{CLMS Pacific} we have included a brief account of the main advances towards proving this conjecture, see \cite{BACH, AVRAMOVVIGUE, BERGHMADSEN2009, BERGHMADSEN2017, BERGHERDMANN, SOLOTARVIGUE, SOLOTARSUAREZVIVAS, CIBILSREDONDOSOLOTAR}.

An extension $B\subset A$ of $k$-algebras is called  left (resp. right) bounded if the $B$-bimodule $A/B$ is $B$-tensor nilpotent, its projective dimension is finite and it is left (resp. right) $B$-projective. A main result of this paper is that for a  left or right bounded extension $B\subset A$, the $k$-algebra $B$ satisfies Han's conjecture if and only if $A$ does. Notice that we do not assume the extension to be split.

Actually the purpose of this paper is twofold. Firstly to provide results which may be used to associate to a given algebra $A$ an easier algebra $B$ for which the property of satisfying Han's conjecture is equivalent to the same property for $A$. In \cite{CIBILSREDONDOSOLOTAR,CLMS Pacific} a theory is developed for split extensions. Here we obtain results that allow us to deal with non split extensions. Even if there are algebras to which we cannot apply our methods, our guiding idea is that, starting from a given finite dimensional algebra, we can reduce the study of Han's conjecture about it to the same problem for an algebra where the computation of the Hochschild homology is more tractable.  The second aim of the paper is to provide examples of this procedure. We develop a theory of extensions by arrows and relations and we give conditions ensuring that the extension is left or right bounded. We also show that all extensions of an algebra given by a quiver and admissible relations are obtained through our construction in Section \ref{section by arrows and relations}, Theorem \ref{structuretheorem}.

We describe in what follows the structure of the paper in more detail.

Let $B\subset A$ be an extension of $k$-algebras. G. Hochschild defined Hochschild (resp. relative) homology with coefficients in an $A$-bimodule $X$  in \cite{HOCHSCHILD1945} (resp. \cite{HOCHSCHILD1956}). These are vector spaces denoted $H_*(A,X)$ (resp. $H_*(A|B, X)$), for $*\geq 0$. See also the books \cite{LODAY, WEIBEL, WITHERSPOON} for a general account of Hochschild theory.

The proof of the above mentioned main result of this paper relies on the normalised relative bar resolution and the Jacobi-Zariski long nearly exact sequence obtained in \cite{CLMS2020nearlyexactJZ}. We give a brief account of these results in Section \ref{nor and jz}.

It is worth noting that in Sections \ref{HH} and  \ref{smoothness} we consider extensions of algebras which may be infinite dimensional.
We first show that if an extension $B\subset A$ is  left or right bounded then the Hochschild homologies of $A$ and $B$ are isomorphic in large enough degrees. Then we prove in Section \ref{smoothness} that for a left or right bounded extension,  finite global dimension  is preserved. Note that for the latter, the normalised relative bar resolution is also a crucial tool. We end Section \ref{smoothness}  by proving the main quoted result about Han's conjecture for  left or right  bounded extensions of finite dimensional algebras.

In Section \ref{section by arrows and relations} we construct explicit extensions of an algebra $B=kQ/I$ where $Q$ is a quiver and $I$ is an admissible ideal of the path algebra $kQ$. A structure result is proved,  which shows that this construction provides every $B$-finitely generated algebra extension of $B$. In detail, we first add new arrows to $Q$, to obtain a tensor algebra extension $T$ of $B$. Then we add new relations through an ideal $J\subset T$ verifying $B\cap J=0$ to obtain an extension $B\subset A=T/J$. We focus on extensions $B\subset A$ of this sort to give conditions ensuring that the $B$-bimodule $A/B$ is tensor nilpotent and/or projective on one side over $B$. We use relative paths, they are alternate concatenations of new arrows and non-zero paths of $B$.  If every relative cycle contains a $J$-interrupter (see Definition \ref{interrupter}), then $A/B$ is $B$-tensor nilpotent. If the projection of $J$ to the positive part of $T$ is generated as a $B$-bimodule by relative paths ending in the same new arrow, then $A/B$ is left projective as a $B$-module.

Finally in Section \ref{examples} we consider examples of extensions by arrows and relations.  The above criteria enables to prove that they are left or right bounded, and we show that they are not split.

We thank Konstiantyn Iusenko and John MacQuarrie for their very pertinent comments on an earlier version of this article. We thank the referee for an attentive reading and useful suggestions.

\section{\sf Normalised relative bar resolution and the Jacobi-Zariski sequence}\label{nor and jz}

In this section we recall some results from \cite{CLMS2020nearlyexactJZ} that will be useful for the sequel. We first recall the normalised relative bar resolution.

\begin{prop}\cite[Proposition 2.3]{CLMS2020nearlyexactJZ}\label{normalised bar resolution}
Let $B\subset A$ be an extension of $k$-algebras and let $\sigma$ be a $k$-linear section of the canonical projection $\pi: A\to A/B$. The following is a $B$-relative resolution of $A$:
\begin{equation*}\label{nbrr}\cdots\stackrel{d}{\to}
A\otimes_B(A/B)^{\otimes_Bm}\otimes_BA
\stackrel{d}{\to}\cdots\stackrel{d}{\to}A\otimes_BA/B\otimes_BA\stackrel{d}{\to}A\otimes_BA
\stackrel{d}{\to}A\to 0\end{equation*}
where the last $d$ is the product of $A$ and
\begin{align*}
d(a_0\otimes\alpha_{1}\otimes\dots\otimes&\alpha_{n-1}\otimes a_n)=  a_0\sigma(\alpha_1)\otimes\alpha_2\otimes\dots\otimes\alpha_{n-1}\otimes a_n+\\
&\sum_{i=1}^{n-2}(-1)^ia_0\otimes\alpha_{1}\otimes\dots\otimes \pi(\sigma(\alpha_i)\sigma(\alpha_{i+1}))\otimes\dots\otimes\alpha_{n-1}\otimes a_n +\\
&(-1)^{n-1}a_0\otimes\alpha_{1}\otimes\dots\otimes\sigma(\alpha_{n-1})a_n.
\end{align*}
Moreover the differential $d$ does not depend on the choice of the section $\sigma$.
\end{prop}

The normalised relative bar resolution provides a chain complex for computing the relative Hochschild homology as follows.

 \begin{theo} \cite[Corollary 2.4]{CLMS2020nearlyexactJZ}\label{compute relative homology}
  Let $B\subset A$ be an extension of $k$-algebras and let $X$ be an $A$-bimodule. Let $\sigma$ be a $k$-linear section of $\pi:A\to A/B$. The homology of the following chain complex is the relative Hochschild homology $H_*(A|B,X)$
  \begin{equation}\ \ \cdots  \stackrel{b}{\to} X\otimes_{B^e} (A/B)^{\otimes_Bm}\stackrel{b}{\to}\cdots \stackrel{b}{\to} X\otimes_{B^e}A/B \stackrel{b}{\to} X_B\to 0\end{equation}
  where
  \begin{align*}
b(x\otimes\alpha_{1}\otimes\dots\otimes&\alpha_{n})=  x\sigma(\alpha_1)\otimes\alpha_2\otimes\dots\otimes\alpha_{n}+\\
&\sum_{i=1}^{n-1}(-1)^i \ x\otimes\alpha_{1}\otimes\dots\otimes \pi(\sigma(\alpha_i)\sigma(\alpha_{i+1}))\otimes\dots\otimes\alpha_{n}  +\\
&(-1)^{n}\ \sigma(\alpha_n)x\otimes\alpha_{1}\otimes\dots\otimes\alpha_{n-1}.
\end{align*}
The coboundary $b$ does not depend on the choice of the section $\sigma$.
\end{theo}

\begin{defi}
Let $B$ be a $k$-algebra. A $B$-bimodule $M$ is tensor nilpotent if there exists $n$ such that $M^{\otimes_B n}=0$.
\end{defi}

As usual, we write  ``$*>>0$" for ``large enough $*$", that is ``all $*$ larger than some positive integer".

\begin{coro} \label{A/B nilpotent then relative H is zero in large degrees}
Let $B\subset A$ be an extension of $k$-algebras with $A/B$ tensor nilpotent. Let $X$ be an $A$-bimodule. We have $$H_*(A|B,X)=0 \mbox{ for } *>>0.$$
  \end{coro}

  \begin{proof}

The complex of Theorem \ref{compute relative homology} is zero in large enough degrees. \qed
  \end{proof}

Next we recall the definition of a long sequence of vector spaces which is exact twice in three, as introduced in \cite{CLMS2020nearlyexactJZ}.
\begin{defi}\label{definition long nearly exact}
A \emph{long nearly exact sequence} is a complex of vector spaces
\begin{equation*}
 \dots \stackrel{\delta}{\to} U_{m} \stackrel{I}{\to} V_{m} \stackrel{K}{\to} W_{m} \stackrel{\delta}{\to} U_{m-1} \stackrel{I}{\to} V_{m-1}\to \dots \stackrel{\delta}{\to} U_{n} \stackrel{I}{\to} V_{n} \stackrel{K}{\to} W_{n}
\end{equation*}
ending at a fixed $n\geq 0$ which is exact at $U_*$ and $W_*$.

Its \emph{gap} is the homology at $V_*$, namely the vector space $(\Ker K/\Im I)_*.$
\end{defi}

A. Kaygun has first obtained the Jacobi-Zariski long exact sequence in the non-commutative setting for a $k$-algebra extension $B\subset A$ where $A/B$ is flat as a $B$-bimodule, see \cite{KAYGUN,KAYGUNe}. Note that for commutative algebras the Jacobi-Zariski long exact sequence has been established for Andr\'{e}-Quillen homology in 1974, see for instance \cite[p. 61]{ANDRE} or \cite{Iyengar}. The reason for giving the name ``Jacobi-Zariski" to this sequence is explained in \cite[p. 102]{ANDRE}.

We state now the Jacobi-Zariski long nearly exact sequence obtained  in \cite{CLMS2020nearlyexactJZ}.

\begin{theo}\cite[Theorem 4.4 and Theorem 5.1]{CLMS2020nearlyexactJZ}\label{JZ}
Let $B\subset A$ be an extension of $k$-algebras. Let $X$ be an $A$-bimodule. There is a long nearly exact sequence
\begin{align*} \cdots \stackrel{\delta}{\to} H_{m}(B,X) \stackrel{I}{\to} H_{m}(A,X) \stackrel{K}{\to} H_{m}(A|B,X) \stackrel{\delta}{\to}H_{m-1}(B,X) \stackrel{I}{\to}\cdots \\\stackrel{\delta}{\to} H_{1}(B,X) \stackrel{I}{\to} H_{1}(A,X) \stackrel{K}{\to} H_{1}(A|B,X)
\end{align*}
which is exact at $H_{*}(B,X)$ and $H_*(A|B,X).$

Moreover, there is a spectral sequence converging to its  gap  $(\Ker K/\Im I)_*$ for $*\geq 2$.
 If $\Tor_*^B(A/B, (A/B)^{\otimes_B n})=0$ for $*>0$ and for all $n$, its terms at page $1$ are
$$E^1_{p,q}= \Tor^{B^e}_{p+q}(X,(A/B)^{\otimes_Bp}) \mbox{   \ \ for } p,q>0$$
and $0$ elsewhere.
\end{theo}

\begin{defi}\label{bounded}
An extension $B\subset A$ of $k$-algebras is {\emph{left (resp. right) bounded}} if

\begin{itemize}
\item [-]$A/B$ is tensor nilpotent,

\item [-]$A/B$ is of finite projective dimension as a $B^e$-module,

\item [-]$A/B$ is a left {(resp. right)} projective $B$-module.
\end{itemize}

\end{defi}

 \begin{rema}
The hypothesis of Theorem \ref{JZ} for describing the first page of the spectral sequence is fulfilled if $A/B$ is a left or right bounded extension.
\end{rema}

\begin{theo}  \cite[Theorem 6.5]{CLMS2020nearlyexactJZ} \label{exactJZ}

Let $B\subset A$ be a left or right bounded extension of $k$-algebras and let $X$ be an $A$-bimodule.

There is a Jacobi-Zariski long exact sequence ending at some $n$:
\begin{align*} \dots \stackrel{\delta}{\to} H_{m}(B,X) \stackrel{I}{\to} H_{m}(A,X) \stackrel{K}{\to} H_{m}(A|B,X) \stackrel{\delta}{\to}H_{m-1}(B,X) \stackrel{I}{\to}\dots \\\stackrel{\delta}{\to} H_{n}(B,X) \stackrel{I}{\to} H_{n}(A,X) \stackrel{K}{\to} H_{n}(A|B,X).
\end{align*}

\end{theo}
\section{\sf  Hochschild homology of bounded extensions}\label{HH}
In this section we prove that for any left or right bounded extension $B\subset A$ the Hochschild homology remains unchanged in
large enough degrees.
\begin{theo}\label{HHB HHA}
Let $B\subset A$ be an extension of $k$-algebras, {where the $B$-bimodule $A/B$ is of finite projective dimension and  tensor nilpotent.} There is a  linear injection $$H_*(B,B)\hookrightarrow H_*(A,A) \mbox{ for } *>>0.$$

If {moreover} the extension is left or right bounded then
$$H_*(B,B)\simeq H_*(A,A) \mbox{ for }*>>0.$$
\end{theo}
\begin{proof}
Firstly we assert that if $A/B$ is of finite projective dimension as $B$-bimodule, then  $H_{*}(B,B)$ and  $H_{*}(B,A)$ are isomorphic for $*>>0.$
Indeed, the exact sequence of $B$-bimodules  \begin{equation}\label{primary short exact sequence}
 0\to B\to A\to A/B\to 0.
 \end{equation}
induces a long exact sequence in Hochschild homology
\begin{align*} \cdots {\to} H_{m}(B,B) {\to} H_{m}(B,A) {\to} H_{m}(B,A/B) {\to}H_{m-1}(B,B) {\to}\cdots \\\ \to H_{0}(B,B) {\to} H_{0}(B,A){\to} H_{0}(B,A/B) \to 0.
\end{align*}
By definition for any $B$-bimodule $M$
$$H_*(B,M)=\Tor_*^{B^e}(B,M).$$
Since $M=A/B$ is of finite projective dimension as a $B^e$-module, we infer that $H_*(B, A/B)=0$ for $*>>0.$ The assertion follows.

Next we consider the Jacobi-Zariski long nearly exact sequence of Theorem \ref{JZ} for an $A$-bimodule $X$. It is exact at $H_*(B,X)$, that is $\Ker I =\Im \delta$. Since $A/B$ is $B$-tensor nilpotent, by Corollary \ref{A/B nilpotent then relative H is zero in large degrees} we have that $H_*(A|B, X)=0$ for $*>>0,$ thus for large enough degrees we get $\delta=0$. Hence  the maps $I : H_{*}(B,X) \to H_{*}(A,X)$
are injective  for $*>>0.$ For $X=A$, we infer from the previous assertion that $$H_*(B,B) \hookrightarrow H_*(A,A) \mbox{ for } *>>0.$$
Note that the gap of the Jacobi-Zariski long nearly exact sequence is $$H_*(A,A)/H_*(B,B) \mbox{ \ for } *>>0.$$

If the extension is left or right bounded then by Theorem \ref{exactJZ} the Jacobi-Zariski sequence is exact in degrees larger than some $n$. Hence the gap is zero in large enough degrees and $H_*(A,A)\simeq H_*(B,B) \mbox{ \ for } *>>0.$
\qed
\end{proof}

\section{\sf Bounded extensions, finite global dimension and Han's conjecture}\label{smoothness}
\subsection{\sf Finite global dimension}
We consider extensions $B\subset A$ with the aim of studying the relation between the finiteness of the global dimension $B$ and $A$. For completeness we first recall the following.

\begin{lemm}\label{Aprojective is Bfinite projective dimension}
Let $B\subset A$ be an extension of $k$-algebras,  with $A/B$  of finite projective dimension $s$ as a left $B$-module. Any left projective $A$-module is of  projective dimension at most $s$ as a left $B$-module.

 \end{lemm}

 \begin{proof}
The exact sequence  (\ref{primary short exact sequence})  shows that $A$ has projective dimension  at most $s$ as a left $B$-module. Then the result holds  with the same bound for the projective dimension for any left projective $A$-module by standard arguments. \qed
 \end{proof}

We recall that for noetherian rings the left and right global dimensions coincide, see \cite{AUSLANDER}.

\begin{theo}\label{Asmooth then Bsmooth}
Let $B\subset A$ be an extension of $k$-algebras where $A$ has left finite  global dimension $a$. If $A/B$ is of finite projective dimension $r$ as a $B$-bimodule and  projective as a $B$-right module, then  $B$ has left global dimension at most $r+a$.
\end{theo}
\begin{proof}
Let $Y$ be a left $B$-module. The exact sequence of $B$-bimodules (\ref{primary short exact sequence})
 is right split, hence the sequence of left $B$-modules
 \begin{equation}\label{Y short exact sequence}
 0\to B\otimes_B Y\to A\otimes_B Y\to (A/B)\otimes_B Y\to 0
 \end{equation}
 is exact. Of course $B\otimes_B Y$ is canonically isomorphic to $Y$. In order to prove that $Y$ is of bounded projective dimension  as a left  $B$-module, we will prove that this is the case for $A\otimes_B Y$ and $(A/B)\otimes_B Y$.

Let $P_*\to A\otimes_BY\to 0$ be a  projective resolution of length $a$ of the $A$-module  $A\otimes_BY$. It is an exact sequence of left $B$-modules by restricting the left actions to $B$.

By Lemma \ref{Aprojective is Bfinite projective dimension} each $P_i$ is of projective dimension at most $r$ as a left $B$-module.  We infer that $A\otimes_B Y$ is of projective dimension at most $r+a$ as a left $B$-module.

For $(A/B)\otimes_B Y$, let $Q_*\to (A/B)\to 0$ be a finite projective resolution of the $B$-bimodule $A/B$  of length $r$. Each $Q_i$ is projective as a $B$-bimodule, then it is projective as a right $B$-module. Moreover $A/B$ is projective as a right $B$-module. Hence  $Q_*\otimes_B Y\to (A/B)\otimes_B Y\to 0$ is exact, and it is a finite projective resolution of the left $B$-module $(A/B)\otimes_B Y$  of length at most $r$. As a consequence, the projective dimension of the left $B$-module $Y$ is at most $r+a+1$.\qed

\end{proof}

\begin{theo}\label{B smooth implies A smooth}
Let $B\subset A$ be an extension of $k$-algebras where $B$ has left finite global dimension $b$. Suppose that the $B$-bimodule $A/B$ is tensor nilpotent and let $n$ be the smallest positive integer such that $(A/B)^{\otimes_B n}=0$. Suppose that $A/B$ is projective as a $B$-right module. Then $A$ has  left global dimension at most n+b-1.
\end{theo}

\begin{proof}
First we will consider $A$-modules of the form  $A\otimes_B Y$ where $Y$ is a left $B$-module, to show that they are of projective dimension at most $b$.  Let $P_*\to Y$ be a projective resolution of $Y$ as a $B$-module of length at most $b$. Since the exact sequence (\ref{primary short exact sequence}) is right split, observe that $A$ is projective as a right $B$-module. Hence $A\otimes_B P_* \to A\otimes_B Y\to 0$
 is exact. Moreover  $A\otimes_B P_i$ is projective as a left $A$-module.
 Indeed this holds if $P_i$ is a free left $B$-module, thus it holds for projectives by standard arguments.
  So the projective dimension of the $A$-module $(A\otimes_B Y)$ is at most $b$.

Let now $X$ be an arbitrary $A$-module. We assert that there is an exact sequence of $A$-modules of the form

$$0\to A\otimes_BY_{n-1}\to \cdots\to A\otimes_BY_0\to X\to 0$$
where the $Y_i$'s are left $B$-modules. Note that the assertion shows that $X$ is of  projective dimension  at most $n-1+b$.

To prove the assertion we use that the normalised bar resolution of Theorem \ref{normalised bar resolution} is zero in degrees $\geq n$. We observe that the contracting homotopy of this resolution is given by morphisms of right $A$-modules. Then the following sequence is exact:
\begin{align*}0{\to}
A\otimes_B(A/B)^{\otimes_B n-1}\otimes_BA\otimes_A X
\stackrel{d}{\to}\cdots\stackrel{d}{\to}A\otimes_B(A/B)\otimes_BA\otimes_A X\stackrel{d}{\to}\\
A\otimes_BA\otimes_A X
\stackrel{d}{\to}A\otimes_A X\to 0,\end{align*}
as well as the isomorphic sequence
\begin{align*}0{\to}
A\otimes_B(A/B)^{\otimes_B n-1}\otimes_B X
\stackrel{d}{\to}\cdots\stackrel{d}{\to}A\otimes_B(A/B)\otimes_B X\stackrel{d}{\to}
A\otimes_BX
\stackrel{d}{\to} X\to 0.\end{align*}
The modules $A\otimes_B\left((A/B)^{\otimes_B i}\otimes_B X\right)$ are of the form $A\otimes_BY$. By the above result they are of projective dimension at most $b$ as left $A$-modules. Hence $X$ is of  projective dimension at most $n-1+b$.\qed
\end{proof}

\begin{coro}
Let $B\subset A$ be a right bounded extension of $k$-algebras. The algebra $B$ has left finite global dimension if and only if $A$ has left finite global dimension.
\end{coro}

\begin{rema}
{Analogously, if an extension is left bounded then right finiteness of the global dimension is preserved.}
\end{rema}

\subsection{\sf Han's conjecture for bounded extensions}

Next we use the results that we have obtained to prove the following. {Recall that if a $k$-algebra is finite dimensional then the left and right global dimension are equal. }
\begin{theo}
Let $B\subset A$ be a left or right bounded extension of finite dimensional $k$-algebras.
The algebra $B$ satisfies Han's conjecture if and only if $A$ satisfies Han's conjecture.
\end{theo}

\begin{proof}
Assume that $B$ satisfies Han's conjecture. Suppose $H_*(A,A)=0$ for $*>>0$. Since the extension verifies that the $B$-bimodule $A/B$ is of finite projective dimension and is tensor nilpotent, the first part of Theorem \ref{HHB HHA} gives that $H_*(B,B)=0$ for $*>>0$, {and} then $B$ has finite global dimension. Moreover  the $B$-bimodule $A/B$ is of finite projective dimension and also that $A/B$ is a $B$-left (or right) projective module. Hence by Theorem  \ref{B smooth implies A smooth} If B has finite global dimension then A also have it.

Conversely, assume that $A$ satisfies Han's conjecture and $H_*(B,B)=0$ for $*>>0$. The extension is left or right bounded, hence the second part of Theorem \ref{HHB HHA} gives that $H_*(A,A)=0$ for $*>>0$ and therefore the algebra $A$ has finite global dimension. The $B$-bimodule $A/B$ is of finite projective dimension and {$A/B$} is {projective as} a left or right {projective} module. Theorem \ref{Asmooth then Bsmooth}  ensures that $B$ has finite global dimension. \qed
\end{proof}

\section{\sf Extensions by arrows and relations}\label{section by arrows and relations}

In this section we define extensions of algebras given by arrows and relations, in order to provide conditions which ensure that the extensions are left or right bounded. We also give a structure result, showing that any finitely generated extension is obtained as an extension by arrows and relations.

Let $Q$ be a finite quiver, that is a finite set of vertices $Q_0$, a finite set of arrows $Q_1$ and two maps $s,t:Q_1\to Q_0$ called source and target respectively. A path of length $n>0$ is a sequence of concatenated arrows $b_n\dots b_1$, that is verifying $t(b_i)=s(b_{i+1})$ for $i=1,\dots n-1$. The paths of length zero are the vertices. Let $Q_n$ be the set of paths of length $n$ and consider the path algebra $kQ= \bigoplus_{i\geq 0} kQ_i$. Let $\langle Q_i\rangle$ be the ideal generated by $Q_i$. An ideal $I$ of $kQ$ is \emph{admissible} if for some $n\geq 2$ we have $\langle Q_n\rangle \subset I\subset \langle Q_2\rangle$. Let $B=kQ/I$. It is well known that if $k$ is algebraically closed, any finite dimensional $k$-algebra is Morita equivalent to an algebra of this type.

\begin{defi}\label{new arrows}
Let $Q$ be a quiver. A set $F$ of \emph{new arrows} for $Q$ is a finite set with two maps $s,t:F\to Q_0$. The quiver $Q_F$ is given by $(Q_F)_0=Q_0$ and $(Q_F)_1 =Q_1 \cup F$ with the evident maps $s,t: (Q_F)_1\to Q_0$.
\end{defi}

Let $B=kQ/I$ as above, let $F$ be a set of new arrows for $Q$, and let $\langle I\rangle_{kQ_F}$ be the ideal of $kQ_F$ generated by $I$. Let $T= kQ_F/ \langle I\rangle_{kQ_F}$. Since $\langle I\rangle_{kQ_F}\cap kQ = I$ we have $B\subset T$. It is useful to observe that $T$ is a tensor algebra over $B$ of a projective $B$-bimodule as follows. Let
$$N=\bigoplus_{a\in F} Bt(a)\otimes s(a)B.$$ We have $$T=T_B(N)= B\oplus N \oplus (N\otimes_B N) \oplus \cdots $$

\begin{defi}    \label{by arrows and relations}
An \emph{extension by arrows and relations} $B\subset A$ is given by
\begin{itemize}
  \item $B=kQ/I$ as above,
  \item $F$ a set of new arrows for $Q$,
  \item $J$ an ideal of   \ $T=kQ_F/ \langle I\rangle_{kQ_F}$ verifying $J\cap B=0$.
\end{itemize}
The extended algebra is $A=T/J$.
\end{defi}

\begin{exam}\label{THEexample}
Let $B=kQ/I$ where $Q$ is the full arrows quiver
\[
\begin{tikzcd}
& 2 \arrow[d, "a", dashed, swap] & 3 \arrow[l,"b"]\\
5 \arrow[r,bend left,"\alpha"]  & 1 \arrow[r,"d"] \arrow[l,bend left,swap,"\beta"]  & 4 \arrow[u,"c"]  \\
\end{tikzcd}
\]
and $I=\langle\beta \alpha \rangle$. Let $F=\{a\}$ be the dashed new arrow from $2$ to $1$, and let $J=\langle abcd-\alpha\beta\rangle$. Then $A=kQ_F/\langle \beta\alpha, abcd -\alpha\beta\rangle.$
\end{exam}
\begin{exam}
Let $Q$ be a quiver and $I$ be an admissible ideal of $kQ$. The algebra $\Lambda=kQ/I$ is an extension by arrows and relations of its maximal semisimple subalgebra $E=kQ_0$, where $F=Q_1$ and $J=I$.

In this case the extension $E\subset \Lambda$ is split since $\Lambda= E\oplus r$ where $r$ is the Jacobson radical of $\Lambda$. Moreover it is well known and easy to prove that the $E$-bimodule $\Lambda/E=r$ is $E$-tensor nilpotent if and only if $Q$ has no oriented cycles.
\end{exam}

The next example is a relation extension algebra, see \cite{ASSEMBRUSTLESCHIFFLER,AGST}. The family of relation extension algebras is related with cluster theory.
The example below consists in an algebra $B$ of global dimension 2,  a $B$-bimodule $ M=\Ext^2 (DB, B)$ with trivial multiplication structure. The algebra $A$ is the trivial extension $B\oplus M$.

\begin{exam}\label{rea}See \cite[Example 2.7]{ASSEMBRUSTLESCHIFFLER}  and \cite[Example 5.3]{AGST}.

 Let $B=kQ/I$ where $Q$ is the quiver
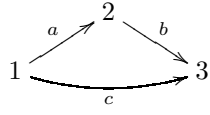
\[\xymatrix@R=10pt{ & 2 \ar[rd]^{b} & \\
1\ar[ru]^{a} \ar@/_/[rr]_{c} &  & 3
}
\]
and $I=\langle ba\rangle$.
Let $F=\{d\}$ be a new arrow $d$ from  $3$ to $1$. The quiver $Q_F$  is
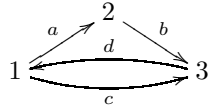
\[\xymatrix@R=10pt{ & 2 \ar[rd]^{b} & \\
1\ar[ru]^{a} \ar@/_/[rr]_{c} &  & 3  \ar@/_5pt/[ll]_{d}
}
\]
Let $J=\langle ad, db, dcd\rangle \subset kQ_F/ \langle ba\rangle_{kQ_F}$. Then $A=kQ_F/ \langle ba, ad, db, dcd\rangle.$
\end{exam}

Let $B\subset A$ be an extension of $k$-algebras. From now on, we call $A$ a $B$-algebra. The following definition is  standard:

 \begin{defi}\label{Bfg}
Let $A$ be a $B$-algebra. A subset $G$ of $A$ is a $B$-generating set for $A$ if every element of $A$ is a sum of products of the form
\begin{equation}\label{alphabet}
b_ng_n\dots b_2g_2b_1g_1b_0
\end{equation}
where $g_i\in G$ and $b_i\in B$ for all $i$.
The extension is called \emph{$B$-finitely generated} if there exists a finite $B$-generating set for $A$.
\end{defi}

Note that if $G$ is a generating set for the $B$-bimodule $A$, then $G$ is a $B$-generating set of the algebra $A$. If $A/B$ is finitely generated as a $B$-bimodule (for instance if $A$ is finite dimensional), then the extension is $B$-finitely generated.

Next we show that any extension of $B$ which is finitely generated over $B$ is an extension by arrows and relations.

Let $B$ be a $k$-algebra and $N$  a $B$-bimodule. The tensor algebra $T_B(N)$ has the following universal property. A morphism of $k$-algebras $\varphi:T_B(N)\to X$ is uniquely determined by a morphism of $k$-algebras $\varphi_B:B\to X$ (which endows $X$ with a $B$-bimodule structure), and a morphism of $B$-bimodules $\varphi_N: N\to X$.

\emph{Let $B$ be a $k$-algebra. Let $BhB$ the free rank one $B$-bimodule $B\otimes B$ whose generator $1\otimes 1$ is denoted $h$. Let $H$ be a set, the \emph{free $B$-bimodule with basis $H$} is denoted $BHB$.}

The proof of the following result is clear.
\begin{prop}\label{tensor epi}
Let $A$ be  a $B$-algebra, $G$  a subset of $A$ and $G'$ a copy of $G$.  Consider the morphism of algebras $\varphi : T_B(BG'B)\to A$ determined by $\varphi_B:B\hookrightarrow A$ and by the $B$-bimodule map $\varphi_{BG'B} : BG'B\to A$  given by $g'\mapsto g$ for all $g\in G$. We have that $G$ is a $B$-generating set of $A$ if and only if $\varphi$ is surjective.

\end{prop}

Let $B=kQ/I$ where $Q$ is a quiver and $I$ an admissible ideal. Let $A$ be a $B$-algebra. The vertices $Q_0$ are a complete set of orthogonal idempotents $A$. The corresponding Peirce decomposition of $A$ is $A=\oplus_{x,y\in Q_0} yAx.$  A subset $F\subset A$ is called a \emph{Peirce subset} if $F= \dot\bigcup_{x,y\in Q_0} F\cap yAx.$

\begin{theo}\label{structuretheorem}
Let $B=kQ/I$ where $Q$ is a quiver and $I$ is an admissible ideal. Every finitely generated $B$-algebra $A$ is an extension by arrows and relations.
\end{theo}

\begin{proof}
Let $G$ be a finite $B$-generating set of the extension $B\subset A$. We first observe that $G$ provides a Peirce set $F= \{ygx\neq 0\mid g\in G\mbox{ and } x,y\in Q_0\}$ which $B$-generates $A$.
Note that  $1\otimes 1=\sum_{x,y\in Q_0} y\otimes x,$ therefore $$B\otimes B=\bigoplus_{x,y\in Q_0} By\otimes xB.$$
Recall that $G'$ is a set in one-to-one correspondence with $G$. Hence for $g\in G$, we have
$$Bg'B= \bigoplus_{x,y\in Q_0} Byg'xB$$
where $Byg'xB$ is the $B$-direct summand of $Bg'B$ which corresponds to $By\otimes xB$ in $B\otimes B$.

  Let $F'$ be a copy of $F$. If $a\in F$ belongs to $yAx$, we set $t(a')=y$ and $s(a')=x$. This way $F'$ is a set of new arrows for $Q$. Let $$N= \bigoplus_{a'\in F'}Bt(a')a's(a')B=\bigoplus_{a'\in F'}Bt(a')\otimes s(a')B.$$
Clearly $N$  is a direct summand of the $B$-bimodule $BG'B$ with complement  $$\bigoplus_{\substack{g\in G\\ygx=0}} Byg'xB.$$
Let $\varphi:T_B(BG'B)\to A$ be the surjective algebra morphism given by Proposition \ref{tensor epi}. Note that if $ygx=0$ then by construction of $\varphi$ we have $$\varphi(yg'x)=y\varphi(g')x=ygx=0.$$
Hence $Byg'xB \subset \Ker\varphi$. Therefore $\psi=\varphi_{\mid_{T_B(N)}}$ has the same image than $\varphi$, then $\psi$ is surjective.

For $J=\Ker \psi$ we infer an isomorphism of algebras $T_B(N)/J \stackrel{\overline{\psi}}{\to} A$. The composition $B\to T_B(N) \stackrel{\psi}{\to} A$ is the inclusion $B\subset A$, hence $J\cap B=0$.

Recall that $B=kQ/I$, and that $Q_{F'}$ is the union of the quiver $Q$ with the new arrows $F'$. As observed before, there is a canonical isomorphism of algebras  $$T_B(N)=\frac{kQ_{F'}}{\langle I\rangle_{kQ_{F'}}}.$$
We still denote by $J$ the ideal in $\frac{kQ_{F'}}{\langle I\rangle_{kQ_{F'}}}$ corresponding to $\Ker \psi$. We have proven that the algebras $\frac{kQ_{F'}}{\langle I\rangle_{kQ_{F'}}}/J$ and $A$ are isomorphic. Moreover the isomorphism is the identity on $B$.\qed

\end{proof}

Let $B\subset A$ be an extension by arrows and relations.  In the sequel we focus on giving a sufficient condition for $A/B$ to be $B$-tensor nilpotent. First we introduce the following definitions.

\begin{defi}
Let $B\subset A$ be an extension by arrows and relations. A \emph{relative path of $F$-length $n\geq 0$}  is a concatenated sequence
$$\beta_na_n\dots\ \beta_2a_2\beta_1a_1\beta_0$$
where the $\beta_i$'s are paths of $Q$ not in $I$ (that is non-zero paths of $B$) and the $a_i$'s are in $F$.
\end{defi}

\begin{rema}\
\begin{itemize}
  \item[-] A relative path of $F$-length $0$ is a path $\beta$ of $Q$ which is non-zero in $B$.
  \item[-] The set of relative paths span $T$ over $k$.
  \item[-] The definition of a relative path only takes into account $kQ/I$ and $F$.
\end{itemize}

  \end{rema}

\begin{defi}
A \emph{relative cycle} of $F$-length $n\geq 1$ is a concatenated sequence
$$\beta_na_n\ \dots\ \beta_2a_2\beta_1a_1$$
where the $\beta_i$'s are paths of $Q$ not in $I$, the $a_i$'s are in $F$, and $t(\beta_n)=s(a_1).$
\end{defi}

\begin{rema}
By abuse of language, we also call a \emph{relative cycle} of $F$-length $n\geq 1$ a concatenated sequence
$$a_1\beta_na_n\ \dots\ \beta_2a_2\beta_1a_1$$
where the $\beta_i$'s are paths of $Q$ not in $I$ and the $a_i$'s are in $F$.
\end{rema}

In Example \ref{THEexample} (resp. \ref{rea}), the sequence $a(bcd)a$ (resp. $dcd$) is a relative cycle of length $1$.

In what follows, for $\gamma\in T$, we denote $\overline{\gamma}$ its class in $A=T/J.$

\begin{defi}\label{interrupter}
Let $a_1\beta_na_n\ \dots\ \beta_2a_2\beta_1a_1$ be a relative cycle. A \emph{$J$-interrupter} of the relative cycle is an arrow $a_i$ such that in case  $2\leq i\leq n$ we have
$$\overline{\beta_ia_i\beta_{i-1}}\in B,$$
while for $i=1$ we have
$$\overline{\beta_1a_1\beta_{n}}\in B.$$
\end{defi}

In Example \ref{THEexample}  the arrow $a$  is a $J$-interrupter of the relative cycle $a(bcd)a$. Indeed
$$\overline{(bcd)a(bcd)}=\overline{ bcd\alpha\beta}\in B.$$

On the other hand, in Example \ref{rea} the arrow $d$ is not a $J$-interrupter of the relative cycle $dcd$ since
$\overline{cdc}\notin B.$ Moreover $A/B$ is not tensor nilpotent over $B$. Indeed $cd\otimes cd = cdc\otimes d $ represents a non zero element of $(A/B)^{\otimes_B2}$, hence $cd\otimes cd \otimes \cdots \otimes cd$ represents a non zero element of  $(A/B)^{\otimes_Bn}$  for all $n$.

\begin{lemm}\label{no relative cycles}
Let $B\subset A$ be an extension by arrows and relations.
If there is no relative cycles then $A/B$ is $B$-tensor nilpotent.
\end{lemm}

Before providing the proof, we underline that an oriented cycle in $Q_F$ is not necessarily a relative cycle, as in the following example.
\begin{exam}
Let $B=kQ/I$ where $Q$ is the full arrows quiver
$$
\begin{tikzcd}
& 2 \arrow[r,"b"] &3\arrow[d,"c"]\\
  &1\arrow[u, "a", dashed, swap]  & 4 \arrow[l,"d"] \\
\end{tikzcd}
$$
and $I=\langle dcb\rangle$. Let $F=\{a\}$ be the dashed new arrow from $1$ to $2$. The quiver $Q_F$ has oriented cycles but there are no relative cycles.
\end{exam}
\textbf{Proof of Lemma \ref{no relative cycles}.}
Let $B=kQ/I$, where $Q$ is a quiver and $I$ is an admissible ideal.

We observe first that $A/B$ is spanned over $k$ by $\{\overline{\overline{\gamma}}\}$, that is by the classes modulo $J$ and $B$ of relative paths $\gamma$ of positive $F$-length.

Hence the set of $m$-tensors over $B$  $$\Gamma= \overline{\overline{\gamma_m}}\otimes_B\dots\otimes_B\overline{\overline{\gamma_1}}$$ spans $(A/B)^{\otimes_B m}$ over $k$, where $\gamma_m,\dots,\gamma_1$ are relative paths of positive $F$-length.  Note that since $kQ_0 \subset B$, if the relative paths are not concatenated then $\Gamma=0$.

Recall that $I$ is admissible, this implies that the paths $\beta$ not in $I$ are of bounded length. Hence there is a finite number $n$ of relative paths of the form $\beta a$.

Let $\Gamma= \overline{\overline{\gamma_{n+1}}}\otimes_B\dots\otimes_B\overline{\overline{\gamma_1}}$ be a $k$-generator of $(A/B)^{\otimes_B n+1}$, where $\gamma_{n+1},\dots,\gamma_1$ are concatenated relative paths of positive $F$-length.

We consider the concatenated sequence $\gamma_{n+1}\dots \gamma_1$ and we reduce it as follows. For each occurrence
$$\cdots a''\beta''\otimes_B\beta'a'\cdots $$
we make the product $\beta''\beta'$ in the sequence.

Suppose that some path $\beta''\beta'$ belongs to $I$. Since
\begin{equation}\label{balancea}
\cdots a''\beta''\otimes_B\beta'a'\cdots = \cdots a''\beta''\beta'\otimes_Ba'\cdots
\end{equation}
we infer that the sequence containing $\beta''\beta'$ is in $\langle I\rangle_{kQ_F}$. Then the corresponding tensorand is zero, and $\Gamma=0$.

On the contrary, in case none of the $\beta''\beta'$ is zero in $B$, we infer a relative path of $F$-length at least $n+1$.  Then there is a repetition of some relative path of the form $\beta a$ and we obtain a relative cycle
$$a \beta_m a_m \dots \beta_2a_2 \beta a$$
for some $m\leq n$.
Therefore, this case does not occur.\qed

\begin{theo}\label{interrupter tensor nilpotent}
Let $B\subset A$ be an extension by arrows and relations, where $B=kQ/I$ and $A=\frac{kQ_F}{\langle I\rangle_{kQ_F}}/{J}$.
If each relative cycle has a $J$-interrupter then $A/B$ is $B$-tensor nilpotent.
\end{theo}
\begin{proof}
If there are no relative cycles then Lemma \ref{no relative cycles} shows that $A/B$ is $B$-tensor nilpotent.

Next we assume that each relative cycle contains a $J$-interrupter. As before, let $\Gamma= \overline{\overline{\gamma_{n+1}}}\otimes_B\dots\otimes_B\overline{\overline{\gamma_1}}$ be a $k$-generator of $(A/B)^{\otimes_B n+1}$, where $\gamma_{n+1},\dots,\gamma_1$ are concatenated relative paths of positive $F$-length. By reducing the concatenated sequence $\gamma_{n+1}\dots \gamma_1$ as before we can assume that it is a relative path. Indeed, otherwise there is a product $\beta''\beta'$ arising from some $\cdots a''\beta''\otimes_B\beta'a'\cdots $  which is in $I$, and by (\ref{balancea}) we have $\Gamma=0$.

Since $n$ is the number of relative paths of the form $\beta a$,  there is a repetition of some relative path $\beta a$, that is we have a relative cycle
\begin{equation}\label{star}
 a \beta_m a_m \dots \beta_2a_2 \beta a.
\end{equation}
 Let $c$ be a $J$-interrupter of it, with $\overline{\beta' c\beta}\in B$.
If in the expression of $\Gamma$ we have that $\beta'$ and $\beta$ belong to different tensorands, for instance $\cdots \beta'c\otimes_B \beta\cdots $, using the equality  $$\cdots \beta'c\otimes_B \beta\cdots = \cdots \beta'c\beta\otimes_B \cdots $$ we obtain an expression of $\Gamma$ which contains, in one of its tensorands, the relative path $\beta'c\beta$. Modulo $J$, we replace $\beta'c\beta$ by a linear combination of paths of $Q$ which are non-zero in $B$. Hence $\Gamma$ is a linear combination of $n$-tensors, each of the summands having one less new arrow. If a summand has a tensorand which is in $B$ (\emph{i.e.} it has no more new arrows), then this summand is zero. For the others, we restart the process. This way we end with $\Gamma=0$.
\qed
\end{proof}

Next we will give conditions ensuring that $A/B$ is projective as a left $B$-module. Of course reversing the arguments in the following proof gives conditions ensuring that  $A/B$ is right $B$-projective.

\begin{lemm}\label{trivial}
Let $C$ be a $k$-algebra and let $X$ and $Y$ be respectively a left and a right $C$-module. Let $Y'$ be a right $C$-subbimodule of $Y$.
The $C$-bimodules $X\otimes \frac{Y}{Y'}$ and $\frac{X\otimes Y}{X\otimes Y'}$ are isomorphic.
\end{lemm}
\begin{proof} Consider the exact sequence $0\to Y'\to Y\to \frac{Y}{Y'}\to 0$ and apply the exact functor $X\otimes -$\qed
\end{proof}

\begin{lemm}
Let $B\subset A$ be an extension by arrows and relations, where $A$ is finite dimensional over $k$.
There exists $n>0$ such that any relative path of $F$-length $m\geq n$ is $J$-equivalent to a linear combination of paths of $F$-length smaller than $n$.
\end{lemm}
\begin{proof}
The classes of relative paths $k$-span $A$, and we extract a finite basis from it. Then $n-1$ is the maximal $F$-length of the relative paths in the basis.\qed
\end{proof}

\begin{defi}For a finite dimensional extension by arrows and relations $B\subset A$, the \emph{length index} of the extension is the smallest $n$ as above.
\end{defi}

Recall that for an extension by arrows and relations $B\subset A$, the algebra $T$ is the tensor algebra
$$T_B(N) = B \oplus \bigoplus_{i>0}N^{\otimes_B i} $$
where $N=\oplus_{a\in F} Bt(a)\otimes s(a)B$. Let $T^{>0}=  \oplus_{i>0}N^{\otimes_B i}$ be its positive ideal, and let $T^{]0,n[}$ be the $B$-bimodule $\oplus_{0<i<n}N^{\otimes_B i}$.

\begin{theo}\label{proj}
Let $B\subset A$ be an extension by arrows and relations of length index $n$. Suppose that the projection $J_{]0,n[}$ of $J$ on $T^{]0,n[}$ is generated as a $B$-bimodule by relations of the form $r= a\sum \lambda_\gamma \gamma$ where $a\in F$, the $\gamma$'s are relative paths concatenated with $a$ and $\lambda_\gamma\in k$.

Then $A/B$ is left $B$-projective.
\end{theo}

Before proving the theorem, we observe that the hypothesis is verified in  Example \ref{THEexample}. Indeed, the ideal $J$ is generated by $abcd-\alpha\beta$. Note that $abcda$ is a relative path of $F$-length two, and $\alpha\beta a$ is a relative path of $F$-length one. In $A$,  we have that $abcda = \alpha\beta a$.  The relative paths of $F$-length $\geq 3$ are zero in $A$. Therefore the length index of the extension is $2$. Moreover $J_{]0,2[}$ is generated by $abcd$ as a $B$-bimodule.
\vskip3mm

\begin{proof}
Let $J_{>0}$ be the projection of $J$ on $T^{>0}$. Since $B\cap J=0$, as $B$-bimodules we have that $$\frac{A}{B}=
\frac{T^{>0}}{J_{>0}}.
$$
Moreover, since the length index is $n$
 $$\frac{T^{>0}}{J_{>0}} = \frac{T^{]0,n[}}{J_{]0,n[}}.$$
The  relative path $a\beta$ corresponds to $t(a)\otimes s(a)\beta=t(a)\otimes \beta\in T$. The $B$-bimodule in $T$ generated by $a\beta$ is $k$-spanned by $\beta''t(a)\otimes \beta\beta'$ where $\beta''$ and $\beta'$ are  paths of $Q$ which are non-zero in $B$.

 We observe that for $i\geq 1$ we have
$$N^{\otimes_Bi}= \bigoplus_{a\in F} Bt(a)\otimes s(a)B\otimes_BN^{\otimes_Bi-1}= \bigoplus_{a\in F} Bt(a)\otimes s(a)N^{\otimes_Bi-1}.$$
Therefore
$$T^{]0,n[}=\bigoplus_{a\in F} Bt(a)\otimes s(a)T^{<n-1}.$$

For an arrow $a$, we consider the $r$'s of type $a$, that is of the form  $r= a\sum \lambda_\gamma \gamma$. Let $r'=\sum \lambda_\gamma \gamma \in T$ and let $Y_a$ be the right submodule of $s(a)T^{< n-1}$ generated by them. Clearly
$$J_{>0}= \bigoplus_{a\in F} Bt(a)\otimes  Y_a.$$
This way
$$\frac{A}{B}
=\bigoplus_{a\in F}\frac{ Bt(a)\otimes s(a)T^{< n-1}}{Bt(a)\otimes  Y_a}.
$$

By  Lemma \ref{trivial}, we obtain finally
$$\frac{A}{B}=\bigoplus_{a\in F}Bt(a)\otimes \frac{ s(a)T^{<n-1}}{Y_a}.$$
Note that $Bt(a)\otimes \frac{ s(a)T^{<n-1}}{Y_a}$ is isomorphic as left $B$-module to the direct sum of $\dim_k \left(\frac{ s(a)T^{< n-1}}{Y_a}\right)$ copies of $Bt(a)$. Moreover $Bt(a)$ is the left projective $B$-module corresponding to the idempotent $t(a)$. Hence $A/B$ is a left projective $B$-module.
\qed
\end{proof}

\section{\sf Examples}\label{examples}

We will use the tools that we have developed in order to consider examples of extensions by arrows and relations. We will study if the extension is left or right bounded and if it is split or not.

We recall that an extension $B\subset A$ is split if there is an ideal $M$ of $A$ such that $A=B\oplus M.$ This is of course equivalent to the fact that the algebra inclusion $B\hookrightarrow A$ splits.

\begin{exam}\sf
We recall the algebra extension provided in Example \ref{THEexample}, where $B=kQ/I$ with $Q$  the full arrows quiver

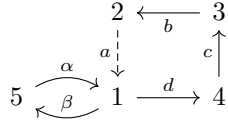
\[
\begin{tikzcd}
& 2 \arrow[d, "a", dashed, swap] & 3 \arrow[l,"b"]\\
5 \arrow[r,bend left,"\alpha"]  & 1 \arrow[r,"d"] \arrow[l,bend left,swap,"\beta"]  & 4 \arrow[u,"c"]  \\
\end{tikzcd}
\]
and $I=\langle\beta \alpha \rangle$. Let $F=\{a\}$ be the dashed new arrow from $2$ to $1$, and let $J=\langle abcd-\alpha\beta\rangle$. Then $A=kQ_F/\langle \beta\alpha, abcd -\alpha\beta\rangle.$

As mentioned, $a$ is   a  $J$-interrupter of the relative cycle $abcda$, and the same holds for longer relative cycles such as $abcdabcda$. So $A/B$ is tensor nilpotent by Theorem \ref{interrupter tensor nilpotent}.
 Let $e_1$ be the idempotent corresponding to the vertex $1=t(a).$  The left $B$-module $A/B$ is projective by Theorem \ref{proj}. More precisely, as noticed before, the length index of the extension is $2$. A basis of $$\frac{ s(a)T^{<1}}{Y_a}=\frac{e_2B}{bcdB}$$ is $\{e_2, b, bc\}$. Hence by  Theorem \ref{proj} we have that $A/B\simeq Be_1\oplus Be_1 \oplus Be_1$ as left $B$-modules.

 The algebra $B$ has finite global dimension, then $B^e$  also have it and the $B$-bimodules are of finite projective dimension, in particular $A/B$ is of finite projective dimension as a $B$-bimodule. Hence this extension by arrows and relations is left bounded.

We claim it is not split. In fact we will prove that there is not even a right $B$-submodule of $A$ such that $A=B\oplus X$. Let $X$ be a right $B$-submodule of $A$. Next we show that $B\cap X\neq 0$.

 Let $\B$ be a basis of paths for $B$ with $Q_0\subset \B$. The vector space $A/B$ is $k$-spanned by relative paths of $F$-positive length, including $a$. Hence we can complete $\B$ to a basis $\A$ of $A$, with relative paths of positive $F$-length, including $a$. Let $D$ be the $k$-vector space with basis $\A\setminus \{a\}$. We claim that there exists $x\in X$ such that in the basis $\A$, the coefficient $x_a$ of $a$ is not $0$. Indeed, otherwise $X\subset D$ and since $B\subset D$ we would have $B\oplus X \subset D$. This cannot hold since $D\subsetneq A$. Then $x=x_aa+y$ with $y\in D$ and $x_a\neq 0$.

Since $X$ is a right $B$-module, we have that $xbcd=x_aabcd+ybcd\in X$. We assert that $ybcd=\lambda bcd$ for $\lambda\in k$. Indeed, $y$ is a $k$-linear combination of the basis $\A\setminus \{a\}$ of $D$. The only possible path of $\A\setminus \{a\}$ which concatenates with $bcd$ on its left is $e_2$, and this proves the assertion.

Then $xbcd = x_a abcd + \lambda bcd = x_a \beta\alpha + \lambda bcd$. Note that $\beta\alpha$ and $bcd$ are linearly independent in $B$. Hence we obtain $0\neq xbcd\in X\cap B$.
\end{exam}

\begin{exam}
\sf
Let $Q$ be the full arrows quiver
$$
\begin{tikzcd}
& 2 \arrow[r,"d"] &3\\
5 \arrow[r,"\mu"]  & 1 \arrow[u, "a", dashed, swap] \arrow[r,"b"] &  \arrow[u,"c"] 4 \\
\end{tikzcd}
$$
and let $B=kQ$. Let $a$ be the dashed new arrow from $1$ to $2$, let $J=\langle da-cb\rangle$ and let $A$ be the corresponding extension by arrows and relations.

We assert that the extension is right bounded and not split. There are no relative cycles, so $A/B$ is tensor nilpotent. Moreover the conditions are satisfied to ensure that $A/B$ is a right projective $B$-module by the dual version of Theorem \ref{proj}. The algebra $B$ is hereditary, hence $B^e$ has finite global dimension and $A/B$ is of finite projective dimension as a $B$-bimodule.
As in the previous example the extension $B\subset A$ is not split.\end{exam}

\footnotesize
\noindent C.C.:\\
Institut Montpelli\'{e}rain Alexander Grothendieck, CNRS, Univ. Montpellier, France.\\
{\tt Claude.Cibils@umontpellier.fr}

\medskip
\noindent M.L.:\\
Instituto de Matem\'atica y Estad\'\i stica  ``Rafael Laguardia'', Facultad de Ingenier\'\i a, Universidad de la Rep\'ublica, Uruguay.\\
{\tt marclan@fing.edu.uy}

\medskip
\noindent E.N.M.:\\
Departamento de Matem\'atica, IME-USP, Universidade de S\~ao Paulo, Brazil.\\
{\tt enmarcos@ime.usp.br}

\medskip
\noindent A.S.:
\\IMAS-CONICET y Departamento de Matem\'atica,
 Facultad de Ciencias Exactas y Naturales,\\
 Universidad de Buenos Aires, Argentina. \\{\tt asolotar@dm.uba.ar}

\end{document}